\documentclass[reqno,oneside]{amsart}

\usepackage[danish,english]{babel} % Danish and English hyphenation etc.
\usepackage[latin1]{inputenc} % Direct input of special characters like æøåÆØÅ
\usepackage[T1]{fontenc} % Hyphenation also with non english characters
\usepackage{times}
\usepackage{amsmath,amsthm,amssymb,amsfonts,mathrsfs,latexsym} % Better mathematics
\usepackage{graphicx} % Include external graphics files
\usepackage{verbatim} % Allows longer comments
\usepackage[all]{xy} % for diagrams
\usepackage[pagebackref,colorlinks, linkcolor=blue, citecolor=blue, urlcolor=blue, hypertexnames=true]{hyperref}  % 

\usepackage{multirow}
\usepackage{appendix}

\makeatletter
\def\namedlabel#1#2{\begingroup
    #2%
    \def\@currentlabel{#2}%
    \phantomsection\label{#1}\endgroup
}
\makeatother

\DeclareFontFamily{U}{mathx}{\hyphenchar\font45}
\DeclareFontShape{U}{mathx}{m}{n}{
      <5> <6> <7> <8> <9> <10>
      <10.95> <12> <14.4> <17.28> <20.74> <24.88>
      mathx10
      }{}
\DeclareSymbolFont{mathx}{U}{mathx}{m}{n}
\DeclareFontSubstitution{U}{mathx}{m}{n}
\DeclareMathAccent{\widecheck}{0}{mathx}{"71}
\DeclareMathAccent{\wideparen}{0}{mathx}{"75}
\renewcommand\check\widecheck
\renewcommand\tilde\widetilde
\makeatletter
\newtheorem*{rep@theorem}{\rep@title}
\newcommand{\newreptheorem}[2]{%
\newenvironment{rep#1}[1]{%
 \def\rep@title{#2 \ref{##1}}%
 \begin{rep@theorem}}%
 {\end{rep@theorem}}}
\makeatother

\newtheorem{thm}[equation]{Theorem}%[section]
\newreptheorem{thm}{Theorem}
\newtheorem{lem}[equation]{Lemma}
\newtheorem{prop}[equation]{Proposition}
\newtheorem{cor}[equation]{Corollary}
\newtheorem*{thm*}{Theorem}

\theoremstyle{definition}

\newtheorem{exam}[equation]{Example}
\newtheorem{rem}[equation]{Remark}
\newtheorem{que}{Question}

\newcommand\mc\mathcal
\newcommand\ms\mathscr
\newcommand\mb\mathbb
\newcommand\mf\mathfrak
\newcommand\mr\mathrm
\newcommand\us\textup

\newcommand{\C}{\mathbb{C}}
\newcommand{\N}{\mathbb{N}}
\newcommand{\Q}{\mathbb{Q}}
\newcommand{\R}{\mathbb{R}}
\newcommand{\Z}{\mathbb{Z}}

\newcommand\la\langle
\newcommand\ra\rangle
\renewcommand\epsilon\varepsilon
\renewcommand\phi\varphi
\renewcommand\rho\varrho
\renewcommand\tilde\widetilde
\renewcommand\hat\widehat
\newcommand\acts\curvearrowright

\DeclareMathOperator{\GL}{GL}
\DeclareMathOperator{\SL}{SL}
\DeclareMathOperator{\UL}{UL}
\DeclareMathOperator{\VN}{VN}
\DeclareMathOperator{\Ind}{Ind}

\numberwithin{equation}{section}
\begin{document}
\selectlanguage{english} % Select either of the sets loaded with babel

\begin{abstract}
We study locally compact groups for which the Fourier algebra coincides with the Rajchman algebra. In particular, we show that there exist uncountably many non-compact groups with this property. Generalizing a result of Hewitt and Zuckerman, we show that no non-compact nilpotent group has this property, whereas non-compact solvable groups with this property are known to exist.

We provide several structural results on groups whose Fourier and Rajchman algebras coincide as well as new criteria for establishing this property.

Finally, we study the relation between groups with completely reducible regular representation and groups whose Fourier and Rajchman algebras coincide. For unimodular groups with completely reducible regular representation, we show that the Fourier algebra may in general be strictly smaller than the Rajchman algebra.
\end{abstract}

%%%%%%%% Titel %%%%%%%%%

\title[Groups whose Fourier algebra and Rajchman algebra coincide]{Groups whose Fourier algebra\\ and Rajchman algebra coincide}
\author{S{\o}ren Knudby}
\address{Mathematisches Institut der WWU M\"unster,
\newline Einsteinstra\ss{}e 62, 48149 M\"unster, Germany.}
\email{knudby@uni-muenster.de}
\thanks{Supported by the Deutsche Forschungsgemeinschaft through the Collaborative Research Centre (SFB 878).}
\date{\today}
\maketitle
%%%%%%%%% Text starts here %%%%%%%%%%
%\tableofcontents
%\newpage

\section{Introduction}
Over the years there has been considerable interest in studying locally compact groups with completely reducible regular representation, that is, locally compact groups whose regular representation decomposes as a direct sum of irreducible representations \cite{MR735532}, \cite{MR509261}, \cite{MR552704}, \cite{MR548088}, \cite{MR2459312}. By the Peter-Weyl theorem, compact groups are examples of such groups. It may come as a surprise that these are not the only ones.  Indeed, in the abelian case it is an easy consequence of the Pontryagin duality theorem that the regular representation of a locally compact abelian group decomposes as a direct sum of irreducible representations if and only if the group is compact.

The study of locally compact groups with completely reducible (also called purely atomic) regular representation is related to the study of certain function algebras associated with the groups. We now describe these algebras (see Section~\ref{sec:prelim} for details).

For a locally compact group $G$, we let $B(G)$ denote the \emph{Fourier-Stieltjes algebra} consisting of the matrix coefficients of strongly continuous unitary representations of $G$. The \emph{Fourier algebra} $A(G)$ is the subalgebra of $B(G)$ consisting of the matrix coefficients of the (left) regular representation $\lambda$. It is always the case that $A(G)\subseteq B(G) \cap C_0(G)$, and often the inclusion is strict. Here $C_0(G)$ denotes the (complex) continuous functions on $G$ vanishing at infinity. The \emph{Rajchman algebra} $B_0(G)$, which is simply defined as the intersection
$$
B_0(G) = B(G)\cap C_0(G),
$$
has recently gained renewed interest (see \cite{MR3071703,MR3211015,MR3473390}). But already in 1966, Hewitt and Zuckerman \cite{MR0193435} showed that for non-compact abelian groups $G$, $A(G)\neq B_0(G)$. This generalized a result of Menchoff \cite{menchoff} from 1916, who showed the same for $G = \Z$.

The main objective of the current paper is to study the inclusion $A(G)\subseteq B_0(G)$ and in particular to study when this inclusion can or cannot be an equality:
\begin{align}\label{eq:AB0}
A(G) = B_0(G).\tag{$\star$}
\end{align}
Our first contribution is a generalization of the result of Hewitt and Zuckerman from abelian groups to nilpotent groups. We thus prove the following.
\begin{repthm}{thm:nilpotent}
If $G$ is a nilpotent, locally compact group then $A(G)\neq B_0(G)$ unless $G$ is compact.
\end{repthm}

The above theorem cannot be generalized to solvable groups. Indeed, Khalil \cite[p.~165]{MR0350330} showed that the $ax+b$ group, which is non-compact and solvable, satisfies $A(G) = B_0(G)$.

Non-compact groups which satisfy $A(G)=B_0(G)$ are generally viewed as exceptional, although several examples appeared recently in \cite{K-fourier} and \cite[Theorem~2.1]{MR2323448}. Our second contribution is to show that there are many such groups. We prove the following.

\begin{repthm}{thm:uncountable}
There exist uncountably many (non-isomorphic) second countable locally compact groups $G$ such that $A(G) = B_0(G)$ and $G$ has no compact subgroups (apart from the trivial group).
\end{repthm}

In \cite{MR0493175}, Fig{\`a}-Talamanca studied the Rajchman algebra in relation to having completely reducible regular representation. He proved that if a unimodular group $G$ satisfies \eqref{eq:AB0}, then the regular representation of $G$ is completely reducible. Subsequently, Baggett and Taylor generalized Fig{\`a}-Talamanca's result to include non-unimodular groups \cite[Theorem~2.1]{MR552704}. They proved
\begin{thm*}[\cite{MR552704}]\label{thm:Baggett+Taylor}
If $A(G) = B_0(G)$ for a second countable locally compact group $G$, then the regular representation of $G$ is completely reducible.
\end{thm*}

At some point, people speculated that the converse of the above theorem should hold, that is, that all groups with completely reducible regular representation should satisfy \eqref{eq:AB0}. This is not the case, as was shown by Baggett and Taylor \cite{MR509261}. They produced a non-unimodular group with completely reducible regular representation not satisfying \eqref{eq:AB0}. At the same time Baggett and Taylor suggested that the converse of the above theorem should hold for unimodular groups. Our third contribution is to provide an example of a unimodular group whose regular representation is completely reducible, but where \eqref{eq:AB0} fails, thus supplementing the example from \cite{MR509261} and answering (in the negative) the question about unimodularity raised there.

\begin{repthm}{thm:AR-non-AB0}
There exists a unimodular locally compact group $G$ whose regular representation is completely reducible but nevertheless $A(G)\neq B_0(G)$.
\end{repthm}

Our final, and perhaps most substantial, contribution is of a more structural nature. All groups currently known to satisfy \eqref{eq:AB0} match the conditions of \cite[Theorem~4]{K-fourier} (see Theorem~\ref{thm:irrep-single} below). We show in this paper that there are also groups satisfying \eqref{eq:AB0} that do not match the conditions of that theorem. At the same time, we study how the condition \eqref{eq:AB0} behaves with respect to taking direct products of groups. It is not clear (at least to the author) if the condition \eqref{eq:AB0} is preserved under taking direct products, although we suspect this to be the case. We are, however, able to prove that a finite direct product satisfies \eqref{eq:AB0} provided all the factors are among the groups for which \eqref{eq:AB0} is currently known to hold. We investigate this in Sections~\ref{sec:structure}--\ref{sec:improvement}.

The paper is organized as follows. Section~\ref{sec:prelim} contains some preliminaries on the Fourier and Fourier-Stieltjes algebras of locally compact groups. In Section~\ref{sec:rank-two}, we show that the results from \cite{K-fourier} concerning parabolic subgroups in real rank one simple Lie groups do not generalize to higher rank simple Lie groups. In Section~\ref{sec:uncountable}, we exhibit uncountably many groups satisfying \eqref{eq:AB0}. Sections~\ref{sec:structure}--\ref{sec:improvement} contain the structural results mentioned above. In particular, in Theorem~\ref{thm:irrep-countable} we provide a generalization of \cite[Theorem~4]{K-fourier}. Section~\ref{sec:AR-example} contains our example of a unimodular locally compact group $G$ whose regular representation is completely reducible but nevertheless $A(G)\neq B_0(G)$. We end the paper with some remarks and questions.

%We have collected a few (probably known) facts about $p$-adic matrix groups in an appendix.

\section{Preliminaries}\label{sec:prelim}
To avoid to much repetition we will use the following conventions (although we sometimes partly repeat assumptions for emphasis or clarity): By a group we mean a Hausdorff locally compact group. By a representation of a group we mean a strongly continuous, unitary representation.

Throughout, $G$ will denote a locally compact group. The \emph{Fourier-Stieltjes algebra} $B(G)$ was introduced by Eymard in \cite{MR0228628} to which we refer for details. The algebra $B(G)$ can be described as the matrix coefficients of continuous unitary representations,
$$
B(G) = \{ \la \pi(\cdot)x,y \ra \mid \pi \text{ is a representation of } G \text{ on } H_\pi \text{ and } x,y\in H_\pi \}.
$$
For a function $\phi\in B(G)$, the norm $\|\phi\|_B$ is defined as the infimum $\inf\{\|x\|\|y\|\}$, where the infimum is taken over all representations $(\pi,H_\pi)$ and vectors $x,y\in H_\pi$ such that
$$
\phi(g) = \la \pi(g)x,y\ra, \quad\text{for all } g\in G.
$$
The infimum is attained. With pointwise multiplication $B(G)$ is a unital Banach algebra. The (continuous) positive definite functions are precisely the functions $\phi$ of the form
$$
\phi(g) = \la\pi(g)x,x\ra, \quad\text{for all } g\in G,
$$
and the positive definite functions span $B(G)$.

If $C_c(G)$ denotes the continuous complex functions on $G$ with compact support, then $B(G)\cap C_c(G)$ is an ideal in $B(G)$. Its closure in $B(G)$ is the \emph{Fourier algebra} $A(G)$, which is a closed ideal in $B(G)$. The Fourier algebra can also be described as the coefficient functions of the regular representation $\lambda$ on $L^2(G)$,
$$
A(G) = \{ \la \lambda(\cdot)x,y \ra \mid x,y\in L^2(G) \}.
$$
The norm $\|\ \|_B$ majorizes the uniform norm, and therefore $A(G)\subseteq C_0(G)$. For the same reason, the \emph{Rajchman algebra} $B_0(G) = B(G)\cap C_0(G)$ is also a closed ideal in $B(G)$.

The Rajchman algebra $B_0(G)$ is the linear span of the positive definite functions in $B_0(G)$. This can be seen from \cite[Lemme~2.12]{MR0228628}, since $B_0(G)$ is a translation invariant subspace of $B(G)$.

We will use the notation $\hat G$ for the unitary dual of $G$. Thus, $\hat G$ consists of the irreducible representations of $G$ up to unitary equivalence. When $G$ is second countable, the space $\hat G$ is equipped with the Mackey Borel structure which turns $\hat G$ into a Borel space (see \cite{MR0089999}).

The locally compact group $G$ is of \emph{type~I} if every unitary representation of $G$ generates a type~I von Neumann algebra. Glimm \cite{MR0124756} gave several equivalent formulations of the type~I condition, which can also be found in \cite[\S9.1]{MR0458185}, \cite[\S9.5]{MR0458185}, and \cite[Theorem~7.6]{MR1397028}. We extract the following:

%For a locally compact group to be of type~I, it is sufficient that every factor representation $\pi$ of $G$ generates a type~I von Neumann algebra.
A locally compact group is of type~I if and only if the image of its universal C$^*$-algebra under any irreducible representation contains the compact operators. When $G$ is second countable, $G$ is of type~I if and only if the Mackey Borel structure on $\hat G$ is standard.

The following theorem is a useful tool to establish the equality $A(G) = B_0(G)$ for type~I groups. We will rely on this theorem in Section~\ref{sec:uncountable}, when we provide an uncountable family of non-compact groups $G$ such that $A(G) = B_0(G)$.

\begin{thm}[\cite{K-fourier}]\label{thm:irrep-single}
Let $G$ be a second countable locally compact group. Suppose $G$ is of type~I and satisfies the following condition:
\begin{enumerate}
	\item[\namedlabel{cond:irrep-single}{I}.]  There is a non-compact, closed subgroup $H$ of $G$ such that every irreducible unitary representation of $G$ is either trivial on $H$ or is a subrepresentation of the left regular representation of $G$.
\end{enumerate}
Then
$$
A(G) = B_0(G).
$$
\end{thm}

Conversely, to establish $A(G)\neq B_0(G)$ is often easy, since there are already several results dealing with this problem. Here we want to mention two instances of such results.

A group is called an AR-group if its regular representation is completely reducible, that is, if it is a direct sum of irreducible representations. Recall that a second countable group $G$ satisfying $A(G) = B_0(G)$ is an AR-group (\cite[Theorem~2.1]{MR552704}).

A group $G$ is an IN-group if it has a conjugation invariant neighborhood of the identity, that is, a neighborhood $U$ of the identity such that $gUg^{-1} = U$ for all $g\in G$. Taylor proved (see \cite[p.~190]{MR690194}) that an IN-group which is also an AR-group has to be compact (and compact groups are both IN-groups and AR-groups).

Several examples of AR-groups were studied in \cite{MR509261,MR548088}. The examples of non-compact groups with completely reducible regular representation given in \cite{MR548088} are totally disconnected and unimodular, whereas the examples in \cite{MR509261} are connected and non-unimodular. It is no coincidence that no one produced a connected, unimodular example. In fact, the following theorem of Baggett excludes the possibility. The theorem is more or less contained in \cite{MR735532} and was noted by Taylor in \cite{MR2459312}. Combining the remark after \cite[Proposition~1.2]{MR735532} with \cite[Theorem~2.3]{MR735532} one obtains the following.
\begin{thm}[Baggett]\label{thm:Baggett}
If a second countable locally compact group is connected, unimodular and has a completely reducible regular representation, then it is compact.
\end{thm}

%universal group C$^*$-algebra, correspondence between $\hat G$ and $\hat{C^*(G)}$, 

%First example Khalil, generalization $F\rtimes F^*$, $F$ local field, \cite[F.6.2]{MR2415834}, Eymard+Terp \cite[(2.10)]{MR560840}.

%\begin{defi}\label{defi:totally-unimodular}
%A locally compact group $G$ is \emph{totally unimodular}, if every continuous homomorphism $G\to\R$ is trivial.
%\end{defi}
%
%Obvious examples of totally unimodular groups are compact groups and groups $G$ for which $[G,G] = G$. Here, $[G,G]$ denotes the closure of the commutator subgroup in $G$. Recall that the commutator subgroup of $\SL(n,K)$ is $\SL(n,K)$ when $n\geq 2$ and $K$ is a field with at least four elements. In particular, $G = [G,G]$ when $G = \SL(2,K)$ and $K$ is a local field.

\section{Parabolic subgroups in higher rank simple Lie groups}\label{sec:rank-two}
In \cite{K-fourier}, it was shown that the minimal parabolic subgroups in real rank one simple Lie groups satisfy \eqref{eq:AB0}. This generalized Khalil's result on the $ax + b$ group.
%We show in this section that the minimal parabolic subgroup in $\SL(3,\R)$ does not satisfy \eqref{eq:AB0}. Thus, the results from \cite{K-fourier} do not generalize to higher rank simple Lie groups.
In this section we show that the results from \cite{K-fourier} concerning parabolic subgroups in real rank one simple Lie groups do not generalize to higher rank simple Lie groups. The simple Lie group $\SL(3,\R)$ has real rank two, and its minimal parabolic subgroup $G$ consists of the upper triangular matrices in $\SL(3,\R)$ with positive diagonal entries,
\begin{align*}%\label{eq:parabolic}
G = \left\{\begin{pmatrix}
\lambda & a & c \\
0 & \mu & b \\
0 & 0 & \nu
\end{pmatrix}
\middle\vert\ a,b,c\in\R,\ \lambda,\mu,\nu>0,\ \lambda\mu\nu=1
\right\}.
\end{align*}
We will to prove that $G$ does not satisfy \eqref{eq:AB0}. To do so, it suffices to prove that $G$ has no discrete series, i.e., that the regular representation of $G$ has no irreducible subrepresentations, cf. \cite[Theorem~2.1]{MR552704}. The irreducible representations of $G$ can be determined using the Mackey Machine just as in \cite[Example~3]{MR0098328}. The discrete series of $G$ can be determined using \cite[Corollary~11.1]{MR0342641}.

The group $G$ is the semidirect product of the Heisenberg group $N$ and the diagonal subgroup $D$,
$$
N = \left\{\begin{pmatrix}
1 & a & c \\
0 & 1 & b \\
0 & 0 & 1
\end{pmatrix}
\middle\vert\ a,b,c\in\R
\right\},
\qquad
D = \left\{\begin{pmatrix}
\lambda & 0 & 0 \\
0 & \mu & 0 \\
0 & 0 & \nu
\end{pmatrix}
\middle\vert\ \lambda,\mu,\nu>0,\ \lambda\mu\nu=1
\right\}.
$$
To show that $G$ has no discrete series, it suffices to show that orbits in $\hat N$ under the dual action of $D$ with positive Plancherel measure have stabilizers without discrete series. The Plancherel measure on $\hat N$ is supported on the infinite dimensional representations (see e.g. \cite[p.~241]{MR1397028}). It can be shown that the dual action of $D$ on the infinite dimensional representations of $N$ has two orbits and that the stabilizer in $D$ of each orbit is
$$
\left\{\begin{pmatrix}
\lambda & 0 & 0 \\
0 & \lambda^{-2} & 0 \\
0 & 0 & \lambda
\end{pmatrix}
\middle\vert\ \lambda>0
\right\}.
$$
Since the stabilizer is isomorphic to $\R_+$, and $\R_+$ has no discrete series, it follows from \cite[Corollary~11.1]{MR0342641} that no irreducible representation of $G$ is in the discrete series. We have thus proved the following.
\begin{thm}\label{thm:parabolic}
Let $G$ be the minimal parabolic subgroup in $\SL(3,\R)$. The regular representation of $G$ has no irreducible subrepresentations. In particular, $A(G) \neq B_0(G)$.
\end{thm}

\section{Uncountably many groups satisfying \texorpdfstring{\eqref{eq:AB0}}{(*)}}\label{sec:uncountable}
In this section we show that there exist uncountably many non-compact groups satisfying \eqref{eq:AB0}. Our family of examples consists of 4-dimensional, simply connected, solvable Lie groups.
For each $r\in\R$, define groups $H_r$ and $G_r$ by
$$
H_r = \left\{ \begin{pmatrix}
	a^r & b \\
	0 & a
\end{pmatrix} \middle\vert\
a > 0,\ b\in\R \right\},
$$

$$
G_r = \left\{\begin{pmatrix}
	a^r & b & x \\
	0 & a & y \\
	0 & 0 & 1
\end{pmatrix} \middle\vert\
a > 0,\ b,x,y\in\R \right\}.
$$
Note that $G_r\simeq\R^2\rtimes H_r$, where $H_r$ acts on $\R^2$ by matrix multiplication. The group $G_r$ contains the Heisenberg group $H$ and can also be viewed as a semidirect product $H\rtimes\R_+$ in an obvious way, where the group $\R_+$ then acts on $H$ by dilations. The groups $G_r$ were previously considered in \cite{MR1676245} (with a different parametrization), and it was shown in \cite[Section~4]{MR1676245} that when $0 \leq r \leq 2$, the groups $G_r$ are mutually non-isomorphic. We show here that the groups $G_r$ satisfy \eqref{eq:AB0} when $r\neq 0$.

\begin{thm}\label{thm:uncountable}
With $0 < r \leq 2$, the locally compact groups $G_r$ are mutually non-isomorphic and satisfy $A(G_r) = B_0(G_r)$. Moreover, the only compact subgroup of $G_r$ is the trivial group.
\end{thm}
\begin{proof}
It is easy to see that $G_r$ has no compact subgroups. Indeed, since $H$ and $\R$ have no compact subgroups, neither does $G_r = H\rtimes\R$. In order to show $A(G_r) = B_0(G_r)$ we will apply Theorem~\ref{thm:irrep-single}.

The dual of $G$ can be determined using the Mackey Machine (see e.g. \cite[Theorem~6.42]{MR1397028} and \cite[Chapter~4]{MR3012851}). The dual action $H_r\acts\hat\R^2$ is easily identified as
$$
\begin{pmatrix}
	x\\y
\end{pmatrix}
\mapsto
\begin{pmatrix}
	a^{-r} & 0 \\
	-b a^{-r-1} & a^{-1}
\end{pmatrix}
\begin{pmatrix}
	x\\y
\end{pmatrix},
$$
and the orbits under this action are
\begin{alignat*}{2}
&\mc O_1 &&= \{ (x,y) \mid x > 0 \},
\displaybreak[0] \\
&\mc O_2 &&= \{ (x,y) \mid x < 0 \},
\displaybreak[0] \\
&\mc O_3 &&= \{ (0,y) \mid y > 0 \},
\displaybreak[0] \\
&\mc O_4 &&= \{ (0,y) \mid y < 0 \},
\displaybreak[0] \\
&\mc O_5 &&= \{ (0,0) \}.
\end{alignat*}
Since there are only five orbits, the action $H_r\acts\hat\R^2$ is certainly regular, and the Mackey Machine applies. The stabilizer in $H_r$ of points from orbits $\mc O_1$ and $\mc O_2$ is trivial, and the corresponding irreducible representations of $G_r$ (induced from an element of the orbit) are subrepresentations of the regular representation of $G_r$ (see e.g. \cite{MR509261}). We claim that the remaining irreducible representations of $G_r$ are trivial on the non-compact subgroup
$$
N_1 = \left\{\begin{pmatrix}
	1 & 0 & x \\
	0 & 1 & 0 \\
	0 & 0 & 1
\end{pmatrix} \middle\vert\
x\in\R \right\}.
$$
Note first that $N_1$ is normal in $G_r$. A character $\nu\in\hat\R^2$ from one of the orbits $\mc O_3$ or $\mc O_4$ is trivial on $N_1$. The stabilizer in $H_r$ of $\nu$ is the group
$$
H_\nu = \left\{ \begin{pmatrix}
	1 & b \\
	0 & 1
\end{pmatrix} \middle\vert\
b\in\R \right\},
$$
and the group $\R^2\rtimes H_\nu$ is also normal in $G_r$. An irreducible representation of $G_r$ arising from the character $\nu$ is an induced representation of the form
$$
\pi = \Ind_{\R^2\rtimes H_\nu}^{G_r}(\nu^*\otimes\sigma^*),
$$
where $\nu^*$ is the extension of $\nu$ to $\R^2\rtimes H_\nu$ obtained so $\nu^*$ is trivial on $H_\nu$, and $\sigma^*$ is the extension of some $\sigma\in\hat{ H_\nu}$ to $\R^2\rtimes H_\nu$ obtained so $\sigma^*$ is trivial on $\R^2$. Since $\nu^*\otimes\sigma^*$ is trivial on $N_1$, so is $\pi$ (see e.g. \cite[Lemma~11]{K-fourier}).

Irreducible representations of $G_r$ arising from the trivial character in $\mc O_5$ are precisely those that factorize to representations of $H_r$, and these are clearly trivial on $N_1$.

Note that all stabilizer subgroups of the action $H_r\acts\R^2$ are of type I, so that \cite[Theorem~9.3]{MR0098328} implies that $G_r$ is also of type I.

We have now verified the conditions of Theorem~\ref{thm:irrep-single} and may conclude $A(G_r) = B_0(G_r)$.
\end{proof}

\begin{rem}
It is also true that the groups $H_r$ satisfy $A(H_r) = B_0(H_r)$, except when $r = 1$. However, the groups $H_r$ are all isomorphic, when $r\neq 1$. Indeed, $H_r$ is then a non-abelian, simply connected Lie group of dimension 2, and it is well-known that there is only one such group, because there is only one non-abelian Lie algebra of dimension 2. The group $H_r$ is isomorphic to the $ax +b$ group studied by Khalil.
\end{rem}

\begin{rem}
In \cite[Theorem~2]{K-fourier}, the group $G_{r}$ with $r = -1$ was also shown to satisfy $A(G_r) = B_0(G_r)$. This group is isomorphic to $G_{1/2}$ (see \cite[Section~4]{MR1676245}).
\end{rem}

\section{Structural results and nilpotent groups}\label{sec:structure}
This section is devoted to the study of how the condition \eqref{eq:AB0} behaves under various group constructions (passing to subgroups, taking quotients, etc.). As an application, we prove that \eqref{eq:AB0} never holds for non-compact nilpotent groups (see Theorems~\ref{thm:nilpotent}).

The condition \eqref{eq:AB0} does not behave well under most group constructions, as can be seen from the known examples. For instance, the $ax+b$ group $\R\rtimes\R_+$ satisfies \eqref{eq:AB0}, whereas both the closed normal subgroup $\R$ and the quotient $\R_+$ do not satisfy \eqref{eq:AB0}. Concerning quotient groups, we however have the following easy but useful lemma.

\begin{lem}\label{lem:mod-compact}
Let $G$ be a locally compact group with a compact normal subgroup $K$. If $A(G) = B_0(G)$ then also $A(G/K) = B_0(G/K)$.
\end{lem}
\begin{proof}
Suppose $A(G) = B_0(G)$ and let $\phi\in B_0(G/K)$ be given. Composing with the quotient map $\pi\colon G\to G/K$, we obtain the function $\phi\circ\pi$ in $B_0(G) = A(G)$. Since $\phi\circ\pi$ is obviously constant on $K$-cosets, it follows from \cite[(2.26) and (3.25)]{MR0228628} that $\phi\in A(G/K)$. This finishes the proof.
\end{proof}

We suspect that the converse of Lemma~\ref{lem:mod-compact} is also true.

%Let $G$ be a locally compact group with a compact normal subgroup $K$. Let $\mu$ denote the normalized Haar measure on $K$. For a uniformly continuous function $\phi\colon G\to\C$ define a continuous function $\phi^K$ on $G/K$ by
%$$
%\phi^K(gK) = \int_K \phi(gk)\, d\mu(k), \qquad g\in G.
%$$
%Then we have the following, which we will used later.

%\begin{lem}
%Keep the notation just introduced. The following holds.
%\begin{itemize}
%	\item If $\phi\in B(G)$ then $\phi^K\in B(G/K)$.
%	\item If $\phi$ is positive definite, so is $\phi^K$.
%	\item If $\phi\in C_0(G)$ then $\phi^K\in C_0(G/K)$.
%\end{itemize} 
%\end{lem}
%\begin{proof}
%\end{proof}

We now turn to subgroups. Let $G$ be locally compact group with a closed subgroup $H$. Following \cite{MR3473390}, we say that $H$ is $B_0$-extending in $G$ provided that every function in $B_0(H)$ has an extension to a function in $B_0(G)$. In other words, $H$ is $B_0$-extending if the restriction map $B_0(G) \to B_0(H)$ is surjective. For example, it is obvious that open subgroups are $B_0$-extending. It was proved in \cite[Theorem~4.3]{MR3071703} that if $G$ is a SIN-group (if the identity admits a basis of conjugation invariant neighborhoods) then any closed subgroup is $B_0$-extending.

\begin{lem}
Let $G$ be a locally compact group such that $A(G) = B_0(G)$, and let $H$ be a closed subgroup of $G$. Then $A(H) = B_0(H)$ if and only if $H$ is $B_0$-extending .
\end{lem}
\begin{proof}
Every element of $A(G)$ restricts to an element of $A(H)$, that is, $\phi\in A(G)$ implies ${\phi|_H\in A(H)}$. Moreover, Herz' restriction theorem \cite[Theorem~1b]{MR0355482} says that every element of $A(H)$ is of the form $\phi|_H$ for some $\phi\in A(H)$. It is now trivial to show that $A(H) = B_0(H)$ if and only if $H$ is $B_0$-extending.
\end{proof}

The following was shown by Ghandehari in \cite[Theorem~4.4]{MR3071703} and \cite[Theorem~3.3.5]{Ghandehari-thesis}. A proof of (2) can also be found in \cite[p.~99]{MR0420148} with the additional remark that the proof given there works equally well without changes for functions vanishing at infinity.

%, if one notices that  Since a proof of (2) has not been published, we have chosen to include a proof of that part.
\begin{thm}[Ghandehari]
Let $G$ be a locally compact group. The following subgroups are $B_0$-extending:
\begin{enumerate}
	\item[(1)] any open subgroup of $G$;
	\item[(2)] the identity component of $G$;
	\item[(3)] the center of $G$.
\end{enumerate}
\end{thm}

\begin{cor}
Let $G$ be a locally compact group with an open subgroup $H$.

If $A(G) = B_0(G)$, then $A(H) = B_0(H)$.

If $H$ is moreover cocompact (i.e. has finite index in $G$), then $A(H) = B_0(H)$ implies $A(G) = B_0(G)$.
\end{cor}
\begin{proof}
The first half is a direct consequence of the previous lemma and theorem.

Suppose conversely that $H$ has finite index and $A(H) = B_0(H)$. For a function $\psi\in A(H)$ define $\psi^0$ as the function on $G$ which coincides with $\psi$ on $H$ and is zero on the complement of $H$ in $G$. Then $\psi^0\in A(G)$ (see e.g. \cite[(3.21)]{MR0228628}).

Choose representatives $s_1,\ldots,s_n\in G$ for the left cosets $G/H$. For a function $f$ defined on $G$ and an element $x\in G$, let $L_xf$ denote the function $L_xf(y) = f(x^{-1}y)$. Let $\phi\in B_0(G)$ be given. Then $L_{s_i^{-1}}\phi\in B_0(G)$ and $(L_{s_i^{-1}}\phi)|_H\in B_0(H) = A(H)$. Hence, if we set
$$
\phi_i = L_{s_i} (((L_{s_i^{-1}} \phi)|_H)^0)
$$
we have $\phi_i\in A(G)$. Finally, it is easy to check that $\phi = \sum_i \phi_i\in A(G)$. This shows $B_0(G)\subseteq A(G)$ and completes the proof.
\end{proof}

\begin{cor}
Let $G$ be a locally compact group such that $A(G) = B_0(G)$. If $G_0$ denotes its connected component of the identity then $A(G_0) = B_0(G_0)$. In particular, $G_0$ is either compact or non-unimodular.
\end{cor}
\begin{proof}
The first part is immediate from the lemma and theorem above. Since $G_0$ is connected, the second part is almost an immediate consequence of Baggett's theorem (Theorem~\ref{thm:Baggett}) together with \cite[Theorem~2.1]{MR552704}, except that we have not assumed $G$ to be second countable.

Since $G_0$ is connected, it is also $\sigma$-compact. By the Kakutani-Kodaira Theorem (see \cite[Theorem~8.7]{MR551496}) there is a compact normal subgroup $K\triangleleft G_0$ such that $G_0/K$ is second countable and of course still connected. By Lemma~\ref{lem:mod-compact} we also have $A(G_0/K) = B_0(G_0/K)$, so Baggett's theorem now implies that $G_0/K$ is either compact or non-unimodular, and hence the same is true for $G_0$ (see e.g. \cite[p.~91]{MR0175995}).
\end{proof}

\begin{cor}\label{cor:compact-center}
A locally compact group $G$ satisfying $A(G)=B_0(G)$ has compact center.
\end{cor}
\begin{proof}
From the lemma and theorem above, one can deduce $A(Z) = B_0(Z)$, where $Z$ denotes the center of $G$. Since the center is abelian, it follows (e.g. from \cite[Theorem~5.6]{MR0193435}) that $Z$ is compact.
\end{proof}

One can give a different proof of Corollary~\ref{cor:compact-center} using ideas of Kaniuth, Lau, and Ülger from \cite[Example~2.6.(4)]{MR3473390} together with the fact that the Gelfand spectrum of $A(G)$ is simply $G$ (see \cite[Th\'eor\`eme~3.34]{MR0228628}).

As already mentioned and used, it was shown by Hewitt and Zuckerman in \cite[Theorem~5.6]{MR0193435} that abelian groups satisfying \eqref{eq:AB0} are compact. We show below how to extend their result to nilpotent groups. The proof relies on several of the previous results.

\begin{thm}\label{thm:nilpotent}
If $G$ is a nilpotent, locally compact group then $A(G)\neq B_0(G)$ unless $G$ is compact.
\end{thm}
\begin{proof}
Let $G$ be a nilpotent, locally compact group and suppose $A(G) = B_0(G)$. We show that $G$ is compact. We use induction on the nilpotency length, that is, the length of an upper central series. If the nilpotency length is zero, then $G$ is the trivial group, and there is nothing to prove.

Suppose the nilpotency length $n$ is a least one, and let $Z$ denote the center of $G$. By Theorem~\ref{cor:compact-center}, $Z$ is compact. It therefore follows from Lemma~\ref{lem:mod-compact} that $A(G/Z) = B_0(G/Z)$. The group $G/Z$ has nilpotency length $n-1$. Hence, by the induction hypothesis, $G/Z$ is compact. As $Z$ was also compact, we conclude that $G$ itself is compact.
\end{proof}

It is clear from Khalil's result \cite[p.~165]{MR0350330} on the $ax+b$ group and also from Theorem~\ref{thm:uncountable} that one cannot extend Theorem~\ref{thm:nilpotent} to solvable groups.

%\begin{rem}
For the case of connected groups, let us point out that the above theorem is direct consequences of Baggett's theorem (Theorem~\ref{thm:Baggett}). Indeed, nilpotent groups are always unimodular. But in fact, connected nilpotent groups satisfying \eqref{eq:AB0} are even abelian. This can be seen from Theorem~\ref{thm:nilpotent} together with the fact that compact connected solvable groups are abelian (see e.g. \cite[Proposition~9.4]{MR3114697}). Hence the theorem is mostly interesting for totally disconnected, nilpotent groups. 
%\end{rem}

%\begin{rem}
It was shown in \cite[Theorem~4.3]{MR3473390} that the Fourier and Rajchman algebras of certain nilpotent groups are not only distinct but even have distinct spectra.
%\end{rem}

We now provide an example of a group, where Theorem~\ref{thm:nilpotent} applies. Of course, we are mostly interested in an example, which is not already covered by previously known results. Since nilpotent groups are unimodular, we will give examples of groups which are not almost connected. Due to Taylor's result on IN-groups \cite[p.~190]{MR690194}, our examples will also not be IN-groups.

%Of course, one could simply take a product of an abelian group and some finite non-abelian nilpotent group, e.g., the group $\R \times Q_8$, where $Q_8$ is the quaternion group with 8 elements. However, the group is ``central--by--finite'' (in every reasonable interpretation of the phrase), and we therefore do not consider the example very interesting. Moreover, the group is an IN-group so that already Taylor's result  applies to show that this group does not satisfy \eqref{eq:AB0}. Our next examples is not an IN-group.

\begin{exam}
Consider the Heisenberg group over the $p$-adic field $\Q_p$,
$$
H = \left\{\begin{pmatrix}
	1 & x & z \\
	0 & 1 & y \\
	0 & 0 & 1
\end{pmatrix}\middle\vert\
x,y,z\in\Q_p
\right\}.
$$
The group is two-step nilpotent and totally disconnected.
%Its center consists of
%$$
%Z = \left\{\begin{pmatrix}
%	1 & 0 & z \\
%	0 & 1 & 0 \\
%	0 & 0 & 1
%\end{pmatrix}\middle\vert\
%z\in\Q_p
%\right\},
%$$
%The group $H$ is torsion-free, and its center has infinite index, so the group is not ``central--by--finite'' in any reasonable interpretation of the phrase.
That the group is not an IN-group can been seen as follows. Let $U$ be a compact neighborhood of the identity in $H$. We use the notation
$$
H(x,y,z) =
\begin{pmatrix}
	1 & x & z \\
	0 & 1 & y \\
	0 & 0 & 1
\end{pmatrix}.
$$
As $\Q_p$ is non-discrete, there is $y\neq 0$ such that $H(0,y,0)\in U$. By compactness, there is $C\geq 0$ such that $H(x,y,z)\in U \implies |z|_p\leq C$, where $|\ |_p$ denotes the $p$-adic norm. Now,
$$
H(x,0,0)H(0,y,0)H(x,0,0)^{-1} = H(0,y,xy) \notin U
$$
if $x$ is chosen such that and $|x|_p > C/|y|_p$. Thus, $U$ is not invariant.

One could argue that $H$ has non-compact center, and therefore already Theorem~\ref{cor:compact-center} shows that \eqref{eq:AB0} does not hold for $H$. An example, which in addition has compact center, can be constructed as follows.

Consider the diagonal embedding $\Z[\tfrac{1}{p}]\hookrightarrow \R\times\Q_p$. This is a discrete embedding, and the quotient $\mathbb S_p = (\R\times\Q_p)/\Z[\tfrac{1}{p}]$ is the compact $p$-adic solenoid (see \cite[p.~58]{MR1760253}).
%Let $R$ be an abelian ring with $\Q_p$ as a proper ideal of finite index, e.g. $R = \Q_p \times (\Z/n\Z)$ for $n\geq 2$. Consider the quotient $Z = R/\Q_p$, which is a finite abelian ring.
Let $G$ be the quotient of the group $T$ consisting of upper-triangular $4\times4$-matrices with entries in $\R\times\Q_p$ and $1$ on the diagonal by the central subgroup $T_0$ defined as
$$
T_0 = \left\{\begin{pmatrix}
	1 & 0 & 0 & z \\
	0 & 1 & 0 & 0 \\
	0 & 0 & 1 & 0 \\
	0 & 0 & 0 & 1
\end{pmatrix}\middle\vert\
z\in\Z[\tfrac{1}{p}]
\right\}.
$$

$$
G = T/T_0 = \left\{\begin{pmatrix}
	1 & * & * & z \\
	0 & 1 & * & * \\
	0 & 0 & 1 & * \\
	0 & 0 & 0 & 1
\end{pmatrix}\middle\vert\
*\in \R\times\Q_p,\ z\in \mathbb S_p
\right\}.
$$

Then one may check that $G$ is three-step nilpotent with compact center
$$
Z(G) = \left\{\begin{pmatrix}
	1 & 0 & 0 & z \\
	0 & 1 & 0 & 0 \\
	0 & 0 & 1 & 0 \\
	0 & 0 & 0 & 1
\end{pmatrix}\middle\vert\
z\in \mathbb S_p
\right\}.
$$
The computation
$$
\begin{pmatrix}
	1 & 0 & 0 & 0 \\
	 & 1 & x & 0 \\
	 &  & 1 & 0 \\
	 &  &  & 1
\end{pmatrix}
\begin{pmatrix}
	1 & 0 & 0 & 0 \\
	 & 1 & 0 & 0 \\
	 &  & 1 & y \\
	 &  &  & 1
\end{pmatrix}
\begin{pmatrix}
	1 & 0 & 0 & 0 \\
	 & 1 & x & 0 \\
	 &  & 1 & 0 \\
	 &  &  & 1
\end{pmatrix}^{-1}
=
\begin{pmatrix}
	1 & 0 & 0 & 0 \\
	 & 1 & 0 & xy \\
	 &  & 1 & y \\
	 &  &  & 1
\end{pmatrix}
$$
and the same argument as for the two-step nilpotent group above shows that $G$ is not an IN-group. The group $G$ is of course far from being connected, since $\Q_p$ is totally disconnected. By Theorem~\ref{thm:nilpotent}, one has $A(G) \neq B_0(G)$.

The reason, why our first two-step nilpotent example was not a genuine example where Theorem~\ref{thm:nilpotent} was applicable (because the group had non-compact center), is explained by the following proposition.
\begin{prop}
Any two-step nilpotent group is which is not an IN-group has non-compact center.
\end{prop}
\begin{proof}
Assume $G$ is a two-step nilpotent group with center $Z$. Then $G/Z$ is an abelian group and hence an IN-group. Since one can easily show that compact extensions of IN-groups are IN-groups, it follows that if $Z$ is compact, then $G$ is an IN-group.
\end{proof}
%Useful? \cite[Proposition~9.100]{MR3114697}
\end{exam}

At this point, we shift the focus to AR-groups for a while, that is, to groups whose regular representation is completely reducible. The following lemma is parallel to Lemma~\ref{lem:mod-compact}.

\begin{lem}
Let $G$ be a locally compact group with a compact normal subgroup $K$. If $G$ has completely reducible regular representation, then so does $G/K$.
\end{lem}
\begin{proof}
First note that $G$ is an AR-group if and only if $\VN(G) = \bigoplus_{i\in I} B(H_i)$ for some Hilbert spaces $H_i$. Here, the direct sum means the von Neumann algebra direct sum which consists of all sequences $(T_i)$ with $T_i\in B(H_i)$ such that $\sup_i \|T_i\|< \infty$. 

According to \cite[(3.25)]{MR0228628}, the von Neumann algebra $\VN(G/K)$ is a quotient of $\VN(G)$. Since each $B(H_i)$ is a factor, it follows that $\VN(G/K) = \bigoplus_{i\in J} B(H_i)$ for some subset $J\subseteq I$, and consequently $G/K$ is an AR-group.
\end{proof}

As a first application of the lemma, let us mention that Baggett's theorem also holds without the second countability assumption.

\begin{thm}[Baggett]
If a locally compact group $G$ is connected, unimodular and has a completely reducible regular representation, then $G$ is compact.
\end{thm}
\begin{proof}
Since the group $G$ is connected, it is $\sigma$-compact. By the Kakutani-Kodaira Theorem (see \cite[Theorem~8.7]{MR551496}) there is a compact normal subgroup $K\triangleleft G$ such that $G/K$ is second countable and of course still connected. By the lemma, the group $G/K$ also has a completely reducible regular representation. The group $G/K$ is also unimodular (see e.g. \cite[p.~91]{MR0175995}). Hence $G/K$ is compact by Baggett's theorem. It follows that $G$ itself is also compact.
\end{proof}

In view of \cite{MR552704}, the following also gives a different proof of (and improves) Corollary~\ref{cor:compact-center} -- at least for second countable groups.

\begin{thm}
A locally compact group AR-group has compact center.
\end{thm}
\begin{proof}
Let $G$ be an AR-group. We claim that its center $Z$ is also an AR-group. Let $\lambda_G$ and $\lambda_Z$ be the regular representations of $G$ and $Z$, respectively. It is well-known that the map $\lambda_Z(z) \mapsto \lambda_G(z)$ extends to a $^*$-isomorphism between $\VN(Z)$ and the von Neumann subalgebra of $\VN(G)$ generated by $\{\lambda_G(z) \mid z\in Z\}$ (see e.g. the proof of \cite[Theorem~6]{MR0281843} or use Herz' restriction theorem). In other words, the regular representation of $Z$ and the representation $\lambda_G|_Z$ to $Z$ of the regular representation of $G$ are quasi-equivalent. It therefore suffices to show that $\lambda_G|_Z$ is completely reducible.

Write the regular representation $\lambda_G = \bigoplus_{i\in I} \pi_i$ as a direct sum of irreducible representations $\pi_i$. Let $H_i$ denote the Hilbert space of $\pi_i$. From Schur's lemma we know that $\pi_i(Z)$ is contained in the scalar multiples of the identity $1_{H_i}$. In other words, there is a character $\chi_i\in\hat Z$ such that $\pi_i|_Z \simeq \chi_i \otimes 1_{H_i}$. We therefore have
$$
\lambda_G|_Z = \bigoplus_i \pi_i|_Z = \bigoplus_i \bigoplus_{j=1}^{\dim H_i} \chi_i
$$
which shows that $Z$ is an AR-group. Since the only abelian AR-groups are compact, the proof is complete.
\end{proof}

The following improves Theorem~\ref{thm:nilpotent}.

\begin{thm}\label{thm:nilpotent2}
A nilpotent, locally compact group with completely reducible regular representation is compact.
\end{thm}
\begin{proof}
The proof is completely analogous to the proof of Theorem~\ref{thm:nilpotent}.
\end{proof}

\section{Direct products}\label{sec:products}

In this section we study how the property \eqref{eq:AB0} behaves under direct products. Recall that, for a second countable locally compact group $G$ of type~I, the following condition is sufficient to conclude $A(G) = B_0(G)$ (see Theorem~\ref{thm:irrep-single}).
\begin{enumerate}
	\item[\ref{cond:irrep-single}.]  There is a non-compact, closed subgroup $H$ of $G$ such that every irreducible unitary representation of $G$ is either trivial on $H$ or is a subrepresentation of the left regular representation of $G$.
\end{enumerate}

The main result here is twofold: Firstly, we show by an example that Condition~\ref{cond:irrep-single} is in general not necessary in order to conclude \eqref{eq:AB0} and not stable under products. Secondly, we introduce a weaker Condition~\ref{cond:irrep-countable} which is stable under products and still sufficient to conclude $A(G) = B_0(G)$.

We start out with an example showing that Condition~\ref{cond:irrep-single} is not stable under products (Example~\ref{exam:product1} below). First, let us recall the Kronecker product of representations. For $i=1,2$, let $G_i$ be a locally compact group and let $\pi_i$ be a unitary representation of $G_i$. The \emph{Kronecker product} (also called the external or outer tensor product) is the unitary representation $\pi_1\times\pi_2$ of $G_1\times G_2$ defined by
$$
(\pi_1\times\pi_2)(g_1,g_2) = \pi_1(g_1)\otimes\pi_2(g_2), \qquad g_1\in G_1,\ g_2\in G_2.
$$

For a locally compact group $G$, recall that $\lambda_G$ denotes the regular representation of $G$ on $L^2(G)$. The following is well-known.
\begin{lem}\label{lem:product}
Let $G_1$ and $G_2$ be locally compact groups, and consider their direct product $G_1\times G_2$.
\begin{enumerate}
	\item[(1)] The regular representation $\lambda_{G_1\times G_2}$ is unitarily equivalent to $\lambda_{G_1}\times\lambda_{G_2}$. %holds also for non-second countable groups
	\item[(2)] If at least one of $G_1$ and $G_2$ are of type~I, then $(\pi_1,\pi_2)\mapsto\pi_1\times\pi_2$ defines a bijection $\hat{ G_1}\times\hat{ G_2}\to\hat{G_1\times G_2}$.
	\item[(3)] The group $G_1\times G_2$ is of type~I if and only if both $G_1$ and $G_2$ are of type~I.
\end{enumerate}
\end{lem}
\begin{proof}
\mbox{}

(1)
The unitary operator $L^2(G_1)\otimes L^2(G_2)\to L^2(G_1\times G_2)$ sending $f\otimes g$ to $f\times g$ intertwines $\lambda_{G_1\times G_2}$ and $\lambda_{G_1}\times\lambda_{G_2}$.

(2)
This can for instance be found in \cite[Theorem~7.25]{MR1397028}.

(3)
It is clear that quotients of type~I groups are again of type~I. The converse, that direct products of type~I groups are of type~I, can be found in \cite[p.~200]{MR0056611} in the case of second countable groups. We also provide an alternative proof of this.

The universal C$^*$-algebra of a group $G$ is denoted $C^*(G)$. Recall that there is a natural correspondence between irreducible representations of $G$ and of $C^*(G)$. To see that ${G_1\times G_2}$ is of type~I, recall that a group is of type~I if and only if the image of its universal C$^*$-algebra  under any irreducible representation contains the compact operators (see \cite[\S9.1]{MR0458185}).

For a Hilbert space $H$, let $\mc K(H)$ denote the compact operators on $H$. If $\pi_i$ is an irreducible representation of $G_i$ on the Hilbert space $H_i$, then $\pi_1\times\pi_2(C^*(G_1\times G_2))$ contains the algebraic tensor product $\mc K(H_1)\otimes \mc K(H_2)$, since $G_1$ and $G_2$ are of type~I. Since this algebraic tensor product is dense in the $\mc K(H_1\otimes H_2)$, and since the image of a representation of a C$^*$-algebra is closed, this shows that $\pi_1\times\pi_2(C^*(G_1\times G_2))$ also contains $\mc K(H_1\otimes H_2)$.

By the first part, every irreducible representation of $C^*(G_1\times G_2)$ is of the form $\pi_1\times\pi_2$, where $\pi_i\in\hat{ G_i}$, so this shows that $G_1\times G_2$ is of type~I.
\end{proof}

\begin{exam}\label{exam:product1}
As an example, to show that Condition~\ref{cond:irrep-single} is not stable under forming direct products, consider the $ax + b$ group,
$$
G = \left\{
\begin{pmatrix}
	a & b \\
	0 & 1
\end{pmatrix} \middle\vert\
a>0,\ b\in\R
\right\}.
$$
The unitary dual $\hat G$ is well-known (see e.g. \cite[Section~6.7]{MR1397028}). It consists of two infinite dimensional representations $\pi^+$ and $\pi^-$ contained in the regular representation $\lambda_G$ and a family of characters $\chi_t$ ($t\in\R$) where $\chi_t(a,b) = a^{it}$. Each of these characters annihilate the non-compact subgroup
$$
H = \left\{
\begin{pmatrix}
	1 & b \\
	0 & 1
\end{pmatrix} \middle\vert\
b\in\R
\right\}.
$$
Clearly, Condition~\ref{cond:irrep-single} is satisfied for the group $G$.

Consider now the group $G\times G$. The irreducible representations $\chi_t\otimes\pi^+$ have non-compact kernels contained in $G\times\{1\}$, whereas the the irreducible representations $\pi^+\otimes\chi_t$ have non-compact kernels contained in $\{1\}\times G$. Therefore none of these representations are subrepresentations of the regular representation $\lambda_{G\times G}$ of $G\times G$. However, it also shows that there is no common non-compact subgroup contained in the intersection of the kernels of irreducible representations not contained in the left regular representation $\lambda_{G\times G}$. This shows that Condition~\ref{cond:irrep-single} is not satisfied for the group $G\times G$, even though Condition~\ref{cond:irrep-single} is satisfied for each of the factors.
\end{exam}

It turns out that the group $G\times G$ nevertheless still satisfies \eqref{eq:AB0}. We will return to this in Example~\ref{exam:product2}. The idea is to weaken Condition~\ref{cond:irrep-single} in Theorem~\ref{thm:irrep-single} and therefore improve the theorem. We thus introduce Condition~\ref{cond:irrep-countable} for a locally compact group $G$:

\begin{enumerate}
	\item[\ref{cond:irrep-countable}.] There is a countable family $\mc H$ of non-compact closed subgroups of $G$ such that each irreducible unitary representation of $G$ is either trivial on some $H\in\mc H$ or is a subrepresentation of the left regular representation of $G$.
\end{enumerate}

Clearly, Condition~\ref{cond:irrep-countable} is weaker then Condition~\ref{cond:irrep-single}. We will prove in Theorem~\ref{thm:irrep-countable} that, for second countable locally compact groups $G$ of type~I, Condition~\ref{cond:irrep-countable} is still sufficient to conclude $A(G) = B_0(G)$.

The important difference between Condition~\ref{cond:irrep-single} and Condition~\ref{cond:irrep-countable} is that Condition~\ref{cond:irrep-countable} is preserved under direct products (of type~I groups), as the following proposition shows.

\begin{prop}\label{prop:product}
Let $G_1$ and $G_2$ be type~I groups satisfying Condition~\ref{cond:irrep-countable}. Then $G_1\times G_2$ is of type~I and satisfies Condition~\ref{cond:irrep-countable}.
\end{prop}%
%\begin{prop}
%Let $G_1$ and $G_2$ be type~I groups and let $G = G_1\times G_2$. Suppose, for $j=1,2$, there is a countable family $\mc H_j$ of non-compact closed subgroups of $G_j$ such that
%\begin{itemize}
%\item
%every irreducible representation of $G_j$ is either trivial on some $H\in\mc H_j$ or is a subrepresentation of the regular representaion of $G_j$.
%\end{itemize}
%Then there is a countable family $\mc H$ of non-compact closed subgroups of $G$ such that
%\begin{itemize}
%\item
%every irreducible representation of $G$ is either trivial on some $H\in\mc H$ or is a subrepresentation of the regular representaion of $G$.
%\end{itemize}
%\end{prop}
\begin{proof}
Set $G = G_1\times G_2$. For $i=1,2$ let $\mc H_i$ be a countable family of non-compact closed subgroups of $G_i$ such that every irreducible representation of $G_i$ is either trivial on some $H\in\mc H_i$ or is a subrepresentation of the regular representation of $G_i$. Define $\mc H$ as the collection of the groups $H_1\times\{1\}$ and $\{1\}\times H_2$ where $H_1\in\mc H_1$ and $H_2\in\mc H_2$. Clearly, $\mc H$ is countable, and every group in $\mc H$ is non-compact and closed in $G$. From Lemma~\ref{lem:product}, every irreducible representation of $G$ is of the form $\pi_1\times\pi_2$ where $\pi_i$ is an irreducible representation of $G_i$ ($i=1,2$). If $\pi_i$ is a subrepresentation of $\lambda_{G_i}$ for both $i=1,2$, then by Lemma~\ref{lem:product}, $\pi$ is a subrepresentation of $\lambda_G = \lambda_{G_1}\times\lambda_{G_2}$. Otherwise, $\pi_1$ (say) is trivial on some $H_1\in\mc H_1$ and $\pi$ is trivial on $H_1\times\{1\}\in\mc H$.
\end{proof}

\section{Coefficient spaces}\label{sec:coefficient-spaces}
The purpose of this section is to prove the following claim, which will be used in the proof of Theorem~\ref{thm:irrep-countable} (or more precisely Lemma~\ref{lem:split}): if a sum of positive definite functions vanishes at infinity then each summand also vanishes at infinity. We feel it is natural to study this problem in the context of von Neumann algebras.

Let $M$ be a von Neumann algebra with predual $M_*$. For a subset $I\subseteq M$, define the annihilator of $I$ inside $M_*$ to be
$$
I_\perp = \{\phi\in M_* \mid \phi(x) = 0 \text{ for every } x\in I\}.
$$
\begin{prop}\label{prop:annihilator}
Let $\pi\colon M\to N$ be a surjective, normal $^*$-homomorphism between von Neumann algebras $M$ and $N$. The map $N_*\to M_*$ defined by $\psi\mapsto\psi\circ\pi$ is an isometric isomorphism of $N_*$ onto $(\ker\pi)_\perp$.
\end{prop}
\begin{proof}
Let $I = \ker\pi$ denote the kernel of $\pi$. Any normal functional $\psi\in N_*$ induces a normal functional $\psi\circ\pi\in M_*$ that annihilates $I$. Since $\pi$ maps the closed unit ball of $M$ onto that of $N$, it is clear that $\|\psi\circ\pi\| = \|\psi\|$.

Conversely, any functional $\phi\in M_*$ that annihilates $I$ induces a well-defined functional $\bar\phi$ on $N$ given by $\bar\phi(\pi(x)) = \phi(x)$, where $x\in M$. It is clear that $\bar\phi\circ\pi = \phi$, so to finish the proof, we just need to show that $\bar\phi$ is normal, i.e., that $\bar\phi\in N_*$.

We show that $\ker\bar\phi$ is weak$^*$ closed, which certainly implies normality of $\bar\phi$. By the Krein-Smulian theorem (see \cite[12.6]{MR768926}), we need only show that $\ker\bar\phi\cap B_N$ is weak$^*$ closed, where $B_N$ denotes the closed unit ball of $N$.

It is clear that $\ker\bar\phi = \pi(\ker\phi)$. Also, $\pi$ maps the closed unit ball $B_M$ in $M$ surjectively onto the closed unit ball $B_N$ in $N$. So
$$
\ker\bar\phi\cap B_N = \pi(\ker\phi) \cap \pi(B_M) = \pi(\ker\phi\cap B_M).
$$
The set $\ker\phi$ is weak$^*$ closed, as $\phi$ is normal. The unit ball $B_M$ is weak$^*$ compact (Banach-Alaoglu's theorem), so by normality of $\pi$ we conclude that $\ker\bar\phi\cap B_N$ is weak$^*$ compact and hence weak$^*$ closed. This completes the proof.
\end{proof}

Let $C^*(G)^{**}$ denote the enveloping von Neumann algebra of the universal group C$^*$-algebra $C^*(G)$ of $G$.

Let $(\pi,\mc H_\pi)$ be a unitary representation of $G$ which we also view as a representation of $C^*(G)$. We denote the image $\pi(C^*(G))$ by $C^*_\pi(G)$ and its weak operator closure by $\VN_\pi(G)$.

Then there is unique normal representation $\tilde\pi\colon C^*(G)^{**}\to B(\mc H_\pi)$ extending $\pi$ and such that $\pi(C^*(G)^{**}) = \VN_\pi(G)$. The kernel of $\tilde\pi$ is a weak$^*$ closed ideal in $C^*(G)^{**}$.
%and hence has a unit $e_\pi$, which is a central projection in $C^*(G)^{**}$. The projection $c_\pi = 1 - e_\pi$ is the \emph{central cover} of $\pi$.

Generalizing Eymard's definition of the Fourier algebra \cite{MR0228628}, Arsac introduced the \emph{coefficient space} $A_\pi$ of a representation $(\pi,\mc H_\pi)$ in \cite{MR0300106,MR0444833}. When $\lambda$ is the regular representation, $A_\lambda$ is the Fourier algebra. In general, the coefficient space $A_\pi$ is defined as the norm closed subspace of $B(G)$ generated by the coefficient functions of $\pi$, i.e., generated by the functions
$$
g\mapsto \la\pi(g)x,y\ra \qquad (g\in G),
$$
where $x,y\in\mc H_\pi$. The space $A_\pi$ can be identified with the predual of the von Neumann algebra $\VN_\pi(G)$ by
$$
\la \pi(f),\phi\ra = \int_G \phi(g)f(g)dg
$$
for $\phi\in A_\pi$ and $f\in L^1(G)$.

\begin{prop}\label{prop:face}
Let $G$ be a locally compact group with a unitary representation $\pi$. If $(\phi_n)_{n\in\N}$ is a sequence of positive definite functions on $G$ such that $\sum_n\phi_n\in A_\pi$, then $\phi_n\in A_\pi$ for every $n$.
\end{prop}
\begin{proof}
The representation $\pi$ extends to a normal representation $\tilde\pi\colon C^*(G)^{**}\to\VN_\pi(G)$. Let $I$ denote the kernel of $\tilde\pi$ inside $C^*(G)^{**}$. Then
$$
I_\perp = \{\phi\in B(G) \mid \la x,\phi\ra = 0 \text{ for every } x\in I\}.
$$
Proposition~\ref{prop:annihilator} identifies $\VN_\pi(G)_*\simeq I_\perp$. We also have the identification of $\VN_\pi(G)_*$ with the subset $A_\pi$ in $B(G)$. Examining the definitions, one checks that the corresponding identification $A_\pi\simeq I_\perp$ is merely equality $A_\pi = I_\perp$. Being a C$^*$-algebra, $I$ is the linear span of its positive elements, and we may also write
$$
A_\pi = \{\phi\in B(G) \mid \la x,\phi\ra = 0 \text{ for every positive } x\in I\}.
$$
Let $x\in I$ be a positive operator. As $\phi_n$ is positive definite, we have $\la x,\phi_n\ra \geq 0$. Now, if $\sum_n\phi_n\in A_\pi$, then
$
\sum_n\la x,\phi_n\ra = \la x,\sum_n\phi_n\ra = 0,
$
and we must have $\la x,\phi_n\ra = 0$ for every $n$. It follows that $\phi_n\in A_\pi$ for every $n$.
\end{proof}

\section{Improving Theorem~\ref{thm:irrep-single}}\label{sec:improvement}

As we saw in Section~\ref{sec:products}, Theorem~\ref{thm:irrep-single} does not suffice to establish \eqref{eq:AB0} for groups such as $G\times G$, when $G$ is the $ax+b$ group. The current section rectifies this problem by establishing an improvement of Theorem~\ref{thm:irrep-single}.

\begin{lem}\label{lem:split}
Let $G$ be a locally compact group. Suppose every unitary representation $\pi$ of $G$ is a sum $\rho\oplus\sigma_1\oplus\sigma_2\oplus\cdots$, where each $\sigma_j$ is trivial on some non-compact, closed subgroup $H_j$ and $\rho$ is a subrepresentation of a multiple of the regular representation. Then $A(G) = B_0(G)$.
\end{lem}
\begin{proof}
Let $\phi\in B(G)$ be a positive definite function, and suppose $\phi\in C_0(G)$. We can write $\phi$ in the form $\phi(g) = \la\pi(g)x,x\ra$ for some representation $(\pi,\mc H_\pi)$ and a vector $x\in\mc H_\pi$. We split the representation $\pi$ as a sum $\pi = \rho\oplus\sigma_1\oplus\sigma_2\oplus\cdots$, according to the assumption. We split $\phi = \phi_0 + \phi_1 + \phi_2 +\cdots$ accordingly, where $\phi_0$ is a coefficient of $\rho$ and $\phi_j$ is a coefficient of $\sigma_j$ ($j\geq 1$).

As $\rho$ is a subrepresentation of a multiple of the regular representation, $\phi_0\in A(G)$. In particular, $\phi_0\in C_0(G)$. It follows that $\sum_{j\geq 1}\phi_j \in C_0(G)$. We claim that $\phi_j\in C_0(G)$ for every $j\geq 1$.

In \cite[Proposition~2.2]{MR3211015}, Jolissaint shows that $B_0(G)$ is the coefficient space $A_{\pi_0}$ of the so-called $C_0$-enveloping representation $\pi_0$ of $G$. The claim is therefore a special case of Proposition~\ref{prop:face}.

Since $\sigma_j$ is trivial on the non-compact, closed subgroup $H_j$, $\phi_j$ must be constant on cosets of $H_j$. As these cosets are all closed and non-compact, and since we have just argued that $\phi_j\in C_0(G)$, we must in fact have $\phi_j = 0$ for every $j$. In conclusion, $\phi = \phi_0 \in A(G)$.

In general, any $\phi\in B_0(G)$ is a linear combination of positive definite functions in $C_0(G)$ (see \cite[Proposition~2.1]{MR3211015}). We have just shown that each of these positive definite functions must belong to $A(G)$, and therefore also $\phi\in A(G)$. This proves the inclusion $B_0(G)\subseteq A(G)$, and the proof is complete.
\end{proof}

\begin{thm}\label{thm:irrep-countable}
Let $G$ be a second countable locally compact group. Suppose $G$ is of type~I and satisfies the following condition:
\begin{enumerate}
	\item[\namedlabel{cond:irrep-countable}{II}.] There is a countable family $\mc H$ of non-compact closed subgroups of $G$ such that each irreducible unitary representation of $G$ is either trivial on some $H\in\mc H$ or is a subrepresentation of the left regular representation of $G$.
\end{enumerate}
Then
$$
A(G) = B_0(G).
$$
\end{thm}

\begin{proof}
We enumerate the groups in $\mc H$ as $\mc H = \{H_1,H_2,\ldots\}$. Since $G$ is of type~I, the unitary dual $\hat G$ is a standard Borel space.
%For each $p\in\hat G$, we let $\pi_p$ denote a representative of the class $p$, and we assume that the choice of representative is made in a measurable way (\cite[Lemma 7.39]{MR1397028}).

First, we show that the left regular representation $\lambda$ of $G$ is completely reducible. Since $\lambda$ acts on a separable Hilbert space and $G$ is of type~I, we may write $\lambda$ as a direct integral,
$$
\lambda = \int_{\hat G}^\oplus m_\pi \pi \ d\nu(\pi),
$$
where $\nu$ is a Borel measure on $\hat G$ and $m_\pi\in\{1,2,\ldots,\infty\}$ (see \cite[Theorem 7.40]{MR1397028}). Let
$$
A_j = \{ \pi\in\hat G \mid \pi(h) = 1 \text{ for all } h\in H_j \}.
$$
It is routine to verify that $A_j\subseteq \hat G$ is a Borel set for the Mackey Borel structure. Since $\lambda$ has no subrepresentation which is trivial on a non-compact subgroup, each $A_j$ must be a $\nu$-null set. Let $B = \hat G \setminus\bigcup_{j=1}^\infty A_j$. Then
$$
\lambda = \int_B^\oplus m_\pi \pi \ d\nu(\pi).
$$
We note that if $\pi\in B$, then by assumption $\pi$ is a subrepresentation of $\lambda$. It follows (e.g. from \cite[Corollary~6]{K-fourier}) that $B$ is countable. Hence $\lambda$ is the direct sum:
$$
\lambda = \bigoplus_{\pi\in B} m_\pi \pi.
$$
Next, let $\sigma$ be an arbitrary unitary representation of $G$. We will show that $\sigma$ decomposes as $\sigma = \rho\oplus\left(\bigoplus_j\sigma_j\right)$, where $\rho$ is contained in a multiple of the regular representation of $G$ and $\sigma_j$ is trivial on $H_j\in\mc H$. By Lemma~\ref{lem:split}, this will prove our theorem.

We reduce to the separable case: Since $\sigma$ is direct sum of cyclic representations, we might as well assume that $\sigma$ is cyclic. As $G$ is second countable, $\sigma$ then represents $G$ on a separable Hilbert space.

Then we may disintegrate $\sigma$,
$$
\sigma = \int_{\hat G}^\oplus n_\pi \pi \ d\mu(\pi),
$$
where $\mu$ is a Borel measure on $\hat G$ and $n_\pi \in\{1,2,\ldots,\infty\}$. Let
\begin{align*}
S_1 &= \{\pi \in\hat G \mid \pi(h) = 1 \text{ for all } h\in H_1 \}, \\
S_{j+1} &= \{\pi \in\hat G \mid \pi(h) = 1 \text{ for all } h\in H_{j+1} \} \setminus\bigcup_{i=1}^j S_i,
\end{align*}
and let $R = \hat G \setminus\bigcup_{j=1}^\infty S_j$, so that we have a partition (into Borel sets)
$$
\hat G = R\cup S_1\cup S_2\cup \cdots.
$$
It follows from our assumptions that $R\subseteq B$. If we define
$$
\sigma_j = \int_{S_j}^\oplus n_\pi \pi \ d\mu(\pi), \qquad \rho = \int_{R}^\oplus n_\pi \pi \ d\mu(\pi),
$$
then we have
$$
\sigma = \rho\oplus\left(\bigoplus_{j=1}^\infty\sigma_j\right).
$$
By construction, $\sigma_j$ is trivial on $H_j$. As $R$ is countable, the integral defining $\rho$ is actually a direct sum, so that $\rho$ is a subrepresentation of 
$$
\bigoplus_{\pi\in R} n_\pi \pi
$$
which in turn is a subrepresentation of $\lambda\oplus\lambda\oplus\cdots$. Hence $\rho$ is a subrepresentation of a multiple of $\lambda$. Lemma~\ref{lem:split} completes the proof, showing that $A(G) = B_0(G)$.
\end{proof}

\begin{exam}\label{exam:product2}
With Theorem~\ref{thm:irrep-countable} at our disposal, we can now finish Example~\ref{exam:product1} and show that the direct product of the $ax+b$ group with itself satisfies \eqref{eq:AB0}. In fact, since the $ax+b$ group is of type~I and satisfies Condition~\ref{cond:irrep-single}, then by Proposition~\ref{prop:product} the direct product of the $ax+b$ group with itself is of type~I and satisfies Condition~\ref{cond:irrep-countable}, and this is sufficient to conclude \eqref{eq:AB0}.
\end{exam}

\begin{rem}
In Theorem~\ref{thm:irrep-countable}, one can not replace the countable family $\mc H$ by an uncountable family: With $G = \R$, every irreducible representation of $G$ (i.e. every character) has a non-compact kernel. However, $A(G)\neq B_0(G)$.
\end{rem}

\section{A unimodular AR-group not satisfying \texorpdfstring{\eqref{eq:AB0}}{(*)}}\label{sec:AR-example}
In \cite{MR509261} Baggett and Taylor gave an example of a connected AR-group not satisfying \eqref{eq:AB0}. Their example is the non-unimodular group $\R^2\rtimes\GL(2,\R)^+$, where $\GL(2,\R)^+$ denotes the $2\times 2$ real matrices with positive determinant. At the same time they suggested that there might not be any examples of unimodular AR-groups not satisfying \eqref{eq:AB0}. However, as we shall see in Theorem~\ref{thm:AR-non-AB0} below, it is possible to find an example of a (disconnected) unimodular AR-group not satisfying \eqref{eq:AB0}. According to Theorem~\ref{thm:Baggett} it is not possible to produce an example which is both connected and unimodular. It is probably not surprising that our example is found among the totally disconnected groups. Our example is inspired by \cite{MR509261}, and one can think of our example as the totally disconnected version of the example from \cite{MR509261}.

Let $\Q_p$ be the $p$-adic field, and for $x\in\Q_p$ let $|x|_p$ denote the $p$-adic norm of $x$. Let $\Z_p = \{x\in\Q_p\mid |x|_p \leq 1\}$ be the $p$-adic integers and $\Z_p^* = \{x\in\Q_p \mid |x|_p = 1\}$ be the $p$-adic units. We assume that the Haar measure on $\Q_p$ is normalized such that $\mu(\Z_p) = 1$. Let $\Q_p^2$ be the $p$-adic plane equipped with the Haar measure arising as the product measure of the Haar measure on $\Q_p$. We will use $\mu$ to denote Haar measure on $\Q_p$ and on the plane $\Q_p^2$.

Let $\UL(2,\Q_p)$ denote the closed subgroup of $\GL(2,\Q_p)$ consisting of matrices whose determinant is a $p$-adic unit, that is, an element of $\Z_p^*$. The example we are after is the group
$$
G = \Q_p^2\rtimes\UL(2,\Q_p),
$$
where $\UL(2,\Q_p)$ acts on $\Q_p^2$ by matrix multiplication. We first establish unimodularity.

\begin{lem}
The group $\Q_p^2\rtimes\UL(2,\Q_p)$ is unimodular.
\end{lem}
\begin{proof}
The determinant map $\UL(2,\Q_p)\to\Z_p^*$ is a surjective homomorphism with kernel the special linear group $\SL(2,\Q_p)$. Since $\SL(2,\Q_p)$ is its own commutator group and since $\Z_p^*$ is compact, it follows that $\SL(2,\Q_p)$ and $\Z_p^*$ are \emph{totally unimodular} in the sense that any continuous homomorphism into $\R$ is trivial. Therefore $\UL(2,\Q_p)$ is totally unimodular. Since $\Q_p^2$ is unimodular, it follows that $\Q_p^2\rtimes\UL(2,\Q_p)$ is unimodular (see \cite[p.89]{MR0175995}).
\end{proof}

As the next step, we analyze the unitary dual of $G$ and prove that $G$ is an AR-group. The unitary dual of $G$ can be determined using the Mackey Machine (see e.g. \cite[Theorem~6.42]{MR1397028} and \cite[Chapter~4]{MR3012851}). Let $N = \Q_p^2$. According to Lemma~\ref{lem:dual-action}, the dual action of $\UL(2,\Q_p)$ on $\hat N$ is under the usual isomorphism $\hat N \simeq \Q_p^2$ given by $g.y = (g^t)^{-1}y$ for $g\in\UL(2,\Q_p)$ and $y\in\Q_p^2$. Given any point $(y_1,y_2)\in\Q_p^2$ with $(y_1,y_2)\neq (0,0)$ it is a simple matter to check that the matrix
\begin{align}\label{eq:cross-section}
\sigma\binom{y_1}{y_2} = \begin{pmatrix}
	y_2 & \frac{y_1}{y_1^2 + y_2^2} \\
	-y_1 & \frac{y_2}{y_1^2 + y_2^2}
\end{pmatrix}
\end{align}
belongs to $\UL(2,\Q_p)$ and (under the dual action) sends $(0,1)$ to $(y_1,y_2)$. We conclude that $\UL(2,\Q_p^2)$ acts transitively on $\hat N \setminus \{0\}$. Thus, the dual action $\UL(2,\Q_p^2)\acts\hat N$ has only two orbits, $\{0\}$ and $\Q_p^2\setminus\{0\}$. In particular, $\UL(2,\Q_p)$ acts regularly on $N$.

The first orbit $\{0\}$ gives rise to the representations in $\hat G$ that annihilate $N$, and these representations are naturally identified with the unitary dual of $\UL(2,\Q_p)$.

We denote the second orbit $\Q_p^2\setminus\{0\}$ by $\mc O$ and choose as a representative of the $\mc O$ the element $\nu = (0,1)$. The stabilizer in $\UL(2,\Q_p)$ of $\nu$ is
$$
H_\nu = \left\{ \begin{pmatrix}
	a & b \\
	0 & 1
\end{pmatrix} \middle\vert\
b\in\Q_p,\ a\in \Z_p^*
\right\}.
$$
Note that $H_\nu$ is isomorphic to the Fell group $\Q_p\rtimes\Z_p^*$. We extend $\nu\in\hat N$ to a character $\nu^*$ defined on $G_\nu = NH_\nu$ by letting $\nu^*$ be trivial on $H_\nu$. The remaining irreducible representations of $G$ are then of the form
\begin{align}\label{eq:irred}
\Ind_{G_\nu}^G (\nu^*\otimes(\rho\circ q)),
\end{align}
where $\rho\in\hat H_\nu$ and $q\colon G_\nu\to H_\nu$ is the quotient map. The regular representation $\lambda_G$ of $G$ is
\begin{align}\label{eq:lambda-decomposition}
\lambda_G = \bigoplus_\N \bigoplus_{j\in\Z} \Ind_{G_\nu}^G (\nu^*\otimes(\rho_j\circ q)),
\end{align}
where $(\rho_j)_{j\in\Z}$ denotes the irreducible representations of $H_\nu$ that occur as subrepresentations of the regular representation of $H_\nu$ (see e.g. \cite{MR0420149}). We thus have the following.

\begin{lem}
The regular representation of the group $\Q_p^2\rtimes\UL(2,\Q_p)$ is completely reducible.
\end{lem}

From now on we will only need to consider one representation of the form \eqref{eq:irred}, namely where $\rho$ is the trivial representation of $H_\nu$. We thus consider the irreducible representation $\pi = \Ind_{G_\nu}^G \nu^*$.

Below we will need an explicit formula for $\pi$. There are several equivalent ways to express induced representations. We have chosen to follow \cite[Realization~III, p.~79]{MR3012851} since it suits our needs well. This expression coincides with the one given in \cite[p.~155]{MR1397028}.

We shall make the following obvious identifications
$$
G/G_\nu \simeq \mc O \simeq \Q_p^2 \setminus \{0\}.
$$
Note that the Haar measure on $\Q_p^2$ restricted to the orbit $\mc O$ is invariant under the action of $\UL(2,\Q_p)$ (see Lemma~\ref{lem:determinant}). Let $\sigma\colon G/G_\nu \to G$ be the continuous section of the quotient map $G\to G/G_\nu$ defined in \eqref{eq:cross-section}. Then $\sigma(y).\nu = y$ when $y\in\mc O$. If we let
$$
\beta(g,y) = \sigma(y)^{-1} g \sigma(g^{-1}.y),
$$
then the representation $\pi$ acts on the space $L^2(G/G_\nu) = L^2(\mc O) = L^2(\Q_p^2)$ as
\begin{align}\label{eq:induction}
(\pi_g f)(y) = \nu^*(\beta(g,y)) f(g^{-1}.y), \quad g\in G,\ f\in L^2(\mc O),\ y\in\mc O.
\end{align}

From now on, let $f$ be the characteristic function of the compact set $C = \Z_p\times \Z_p$, and let
$$
\psi(g) = \la \pi(g)f,f\ra.
$$
Then $\psi$ lies in $B(G)$ by definition. We claim that $\psi\notin A(G)$. If $\psi\in A(G)$, then there would exist $h\in L^2(G)$ such that
$$
\psi(g) = \la\lambda_G(g) h,h\ra \quad\text{for all } g\in G.
$$
As $\pi$ is irreducible, $f$ is a cyclic vector for $\pi$, and the restriction of $\lambda_G$ to the cyclic subspace of $L^2(G)$ spanned by $h$ would then be equivalent to $\pi$. However, we see from the decomposition \eqref{eq:lambda-decomposition} of $\lambda_G$ that $\pi$ is not a subrepresentation of $\lambda_G$. We conclude that $\psi\notin A(G)$.

Next, we aim to show that $\psi$ vanishes at infinity. This will be completed in Proposition~\ref{prop:psi}. Put
$$
\phi(m) = \mu(mC \cap C) \quad\text{for } m\in\UL(2,\Q_p).
$$
To show that $\psi\in C_0(G)$ we first prove the following essential lemma.

\begin{lem}\label{lem:phi}
The function $\phi\colon\UL(2,\Q_p)\to[0,1]$ vanishes at infinity.
\end{lem}
\begin{proof}
Fix $\epsilon > 0$. Consider the compact set $SK\subseteq\UL(2,\Q_p)$ where $K = \GL(2,\Z_p)$ and
$$
S = \left\{
\begin{pmatrix}
	a & b \\
	0 & d
\end{pmatrix}
\in\UL(2,\Q_p) \middle\vert |a|_p \leq\epsilon^{-1},\ |d|_p\leq\epsilon^{-1},\ |b|_p\leq\epsilon^{-1} \right\}.
$$
We claim that if $m\notin SK$ then $|\phi(m)| \leq 2\epsilon$.

Note first that $KC = C$. Since by the Iwasawa decomposition any $m\in\UL(2,\Q_p)$ may be written as $m = bk$ where $b$ is upper triangular and $k\in K$ (see e.g. \cite[Proposition~4.5.2]{MR1431508}), and
$$
\mu(mC \cap C) = \mu(bC\cap C)
$$
we may assume that $m$ is upper triangular,
$$
m = \begin{pmatrix}
	a & b \\
	0 & d
\end{pmatrix}.
$$
We will show that if $m\notin S$ then $|\phi(m)| \leq 2\epsilon$.

In the rest of the proof we abbreviate $|x|_p$ by $|x|$ and we let $\log = \log_p$. For $j\in\Z$, define sets
\begin{align*}
N_j &= \{ x\in\Q_p \mid |x| = p^j \},
\\
K_j &= \{ x\in\Q_p \mid |x| \leq p^j \} = \{0\} \cup\bigcup_{i\leq j} N_j.
\end{align*}

By \eqref{eq:Kj}, we have $\mu(K_j) = p^j$. If we write $N_{ijkl} = m(N_i\times N_j) \cap (N_k\times N_l)$ then we have
\begin{align}\label{eq:sum-ijkl}
\mu(mC\cap C) = \sum_{i,j,k,l\leq 0} \mu(N_{ijkl}).
\end{align}
Suppose $(x,y)\in N_i\times N_j$. As $m(x,y) = (ax+by,dy)$, we see that if $m(x,y)\in N_k\times N_l$ then $|dy| = p^l$. From this we see that $p^j |d| = p^l$ so that $j = l - \log |d|$. In other words,
\begin{align}\label{eq:d-est}
N_{ijkl} \neq \emptyset \implies j = l - \log |d|.
\end{align}
In \eqref{eq:sum-ijkl} we get using \eqref{eq:d-est}
\begin{align*}
\mu(mC\cap C)
&= \sum_{i,k,l\leq 0} \mu(N_{i,l-\log|d|,k,l})
\leq \mu(m(K_0\times K_{-\log |d|}))
\\&= \mu(K_0\times K_{-\log|d|})
= \mu(K_{-\log|d|}) = |d|^{-1}.
\end{align*}
Since $|\det g| = 1$ we further have $\log|a| + \log|d| = 0$. It follows from this and \eqref{eq:d-est} that
\begin{align}\label{eq:a-est}
N_{ijkl} \neq \emptyset \implies l = j - \log |a|,
\end{align}
and as above we deduce
\begin{align*}
\mu(mC\cap C)
= \sum_{i,j,k\leq 0} \mu(N_{i,j,k,j - \log |a|})
\leq \mu(K_{-\log|a|}) = |a|^{-1}.
\end{align*}

We have thus have
\begin{align*}
\mu(mC \cap C) \leq \min\{|a|^{-1}, |d|^{-1}\},
\end{align*}
and if $b = 0$ this proves $|\phi(m)| \leq \epsilon$. We may thus suppose $b\neq 0$. We will then prove
\begin{align}\label{eq:measure}
\mu(mC \cap C) \leq 2\min\{|a|^{-1}, |b|^{-1}, |d|^{-1}\}.
\end{align}

The proof of \eqref{eq:measure} is similar to the above, but it is a bit more involved.

Suppose again $(x,y)\in N_i\times N_j$. As $m(x,y) = (ax+by,dy)$, we see that if $m(x,y)\in N_k\times N_l$ then $|ax+by| = p^k$. Using \eqref{eq:ultrametric2} we see that
$$
k = \max\{i + \log|a|,j+\log|b|\} \quad\text{or}\quad i + \log|a| = j+\log|b|.
$$
It follows that
\begin{align}\label{eq:two-cases}
i \leq -\log|a|,\ j\leq -\log|b| \quad\text{or}\quad i + \log|a| = j+\log|b|.
\end{align}
Hence if $N_{ijkl}\neq \emptyset$ then \eqref{eq:two-cases} holds.
We now estimate \eqref{eq:sum-ijkl} in two parts according to \eqref{eq:two-cases}. The first part, where the equalities $i \leq -\log|a|$ and $j\leq -\log|b|$ hold, is
$$
\sum_{\substack{i,j,k,l\leq 0 \\ i\leq - \log|a| \\ j\leq -\log|b|}} \mu(N_{ijkl})
\leq
\mu(K_{-\log|b|}) = |b|^{-1}.
$$
The second part is estimated as follows. If $i + \log|a| = j+\log|b|$, then
$$
l = j - \log|a| = i - \log|b|
$$
and thus
$$
\sum_{\substack{i,j,k,l\leq 0 \\ l = i - \log|b|}} \mu(N_{ijkl})
\leq
\sum_{i,j,k\leq 0} \mu(N_{i,j,k,i-\log|b|}) \leq |b|^{-1}.
$$
Putting things together we find that
$$
\mu(mC\cap C) \leq 2|b|^{-1}.
$$
We have now established \eqref{eq:measure}, and the proof is complete.
\end{proof}

For use in the proof of Proposition~\ref{prop:psi}, we record the following elementary fact (see \cite[Section \S45]{MR0464128}).

\begin{lem}\label{lem:vanishing-uniformly}
If $X$ is a locally compact Hausdorff space and $L\subseteq C_0(X)$ is compact (in the uniform topology) then the functions in $L$ vanish uniformly at infinity. In other words, for each $\epsilon > 0$ exists a compact set $K\subseteq X$ such that $|f(x)| < \epsilon$ whenever $f\in L$ and $x\in X\setminus K$.
\end{lem}

\begin{prop}\label{prop:psi}
The function $\psi$ vanishes at infinity.
\end{prop}
\begin{proof}
Let $\epsilon > 0$ be given. Write $g\in G$ as $g = xm$ with $x\in\Q_p^2$ and $m\in\UL(2,\Q_p)$. Using \eqref{eq:induction} we compute
\begin{align*}
|\psi(g)|
&= \left| \int_{\mc O} (\pi_g f)(y)\overline{f(y)} \ d\mu(y) \right|
\\&= \left| \int_{\mc O} \nu^*(\beta(g,y)) f(g^{-1}.y) \overline{f(y)} \ d\mu(y) \right|
\\&\leq \int_{\mc O} |f(g^{-1}.y) f(y)| \ d\mu(y)
\\&= \mu(g.C \cap C) = \phi((m^t)^{-1}).
\end{align*}
By Lemma~\ref{lem:phi} there is a compact set $\Omega\subseteq\UL(2,\Q_p)$ such that $|\psi(g)|<\epsilon$ if $g = xm\in G$ and $m\notin\Omega$. We still need to take of the variable $x$. Note that
\begin{align*}
\nu^*(\beta(g,y))
&= \nu(\sigma(y)^{-1} x \sigma(y)) \nu^*(\sigma(y)^{-1} m \sigma(g^{-1}. y))
\\&= (\sigma(y).\nu)(x) \nu^*(\sigma(y)^{-1} m \sigma(m^{-1}. y))
\\&= \la x,y\ra \nu^*(\beta(m,y)).
\end{align*}
If we define
$$
F_m(y) = \nu^*(\beta(m,y)) f(m^{-1}.y) \overline{f(y)}
$$
then $F_m \in L^1(\Q_p^2)$. The Fourier transform $\hat F_m$ satisfies %(modulo conjugation)
$$
\left|\hat F_m(x)\right| = \left|\int F_m(y)\la x,y\ra \ dy \right| = |\psi(xm)|.
$$
The Fourier transform $\hat F_m$ vanishes at infinity by the usual Riemann-Lebesgue lemma, and the Fourier transform is continuous $L^1(\Q_p^2)\to C_0(\Q_p^2)$.

We claim that $m\mapsto F_m$ is continuous $\UL(2,\Q_p)\to L^1(\Q_p^2)$. Indeed, as $\nu^*$, $\beta$, and $f$ are continuous, $F_{m_n}\to F_m$ pointwise if $m_n\to m$ for a sequence $(m_n)$ in $\UL(2,\Q_p)$, and the integrable function $f$ dominates the sequence $(F_{m_n})$. An application of Lebesgue's Dominated Convergence Theorem then gives the continuity of $m\mapsto F_m$.

The set $\{\hat F_m \mid m\in\Omega\}$ is a compact subset of $C_0(\Q_p^2)$. By Lemma~\ref{lem:vanishing-uniformly} there exists a compact set $A\subseteq\Q_p^2$ such that $|\psi(x,m)| < \epsilon$ whenever $m\in\Omega$ and $x\notin A$. Thus if $(x,m)\in G\setminus (A\times\Omega)$ then $|\psi(x,m)| < \epsilon$. This proves that $\psi$ vanishes at infinity.
\end{proof}

Since we have produced an element $\psi$ in $B_0(G)$ not belonging to $A(G)$ we have proved the following.
% In fact, $\psi$ also belongs to $B_\lambda(G)$.
\begin{thm}\label{thm:AR-non-AB0}
Let $G = \Q_p^2\rtimes\UL(2,\Q_p)$. Then $G$ is a unimodular AR-group such that $A(G)\neq B_0(G)$.
\end{thm}

We end this section by discussing how the groups $\R^2\rtimes\GL(2,\R)^+$ and $\Q_p^2\rtimes\UL(2,\Q_p)$ fail to meet the criterion in Theorem~\ref{thm:irrep-single}. This is relevant for obtaining a better understanding of the difference between AR-groups and groups satisfying \eqref{eq:AB0}.

In both cases, the irreducible representations can be found and analyzed using Mackey's theory. The situation is particularly easy, since the groups under consideration are semidirect products $G = N\rtimes H$, where $H$ acts regularly on the abelian normal subgroup $N$. The dual action $H\acts\hat N$ has precisely two orbits: $\{0\}$ and $\hat N\setminus\{0\}$. Irreducible representations of $G$ coming from the orbit $\{0\}$ are trivial on $N$ and thus do not contribute the the regular representation.

The stabilizer $H_\nu$ in $H$ of a point $\nu\in\hat N\setminus\{0\}$ in the second orbit is a group isomorphic to an $ax+b$ group, $\R\rtimes\R_+$ in the first case and $\Q_p\rtimes\Z_p^*$ in the second case. Since such groups are AR-groups, it follows rather easily that $G$ is an AR-group (see \cite{MR509261}). However, the irreducible representations of $G$ arising from the non-zero orbit come in two flavors depending on which irreducible representation of the stabilizer subgroup $H_\nu$ is used in the induction procedure. If a representation from the discrete series of $H_\nu$ is used, then the induced representation will also be in the discrete series of $G$ (see \cite[Corollary~11.1]{MR0342641}). However, if instead a character of $H_\nu$ is used, then the induced representation on $G$ will neither be in the discrete series nor have non-compact kernel (in fact such a representation is faithful).
%
%For $G = \Q_p^2\rtimes \UL(2,\Q_p^2)$ such a representation has kernel contained in $\Z_p^2 \rtimes \{1\}$.
%For $G = \R^2\rtimes\GL(2,\R)^+$ such a representation has trivial kernel.
%
This phenomenon explains, at least intuitively, why $G$ is an AR-group which does not satisfy $A(G)=B_0(G)$.

\section{Some questions}

All examples so far of second countable groups satisfying \eqref{eq:AB0} match the conditions of Theorem~\ref{thm:irrep-countable}. In particular all such groups are of type~I. It would be interesting to find other examples. There are non-type~I groups with completely reducible regular representation (see \cite{MR509261}).
\begin{que}
Do there exist groups $G$ not of type~I satisfying $A(G) = B_0(G)$?
\end{que}

Related to Section~\ref{sec:products}, the following natural question was left open.

\begin{que}\label{que:products}
If $G_1$ and $G_2$ are two groups each satisfying \eqref{eq:AB0}, does the direct product $G_1\times G_2$ satisfy \eqref{eq:AB0}?
\end{que}

Recall that a locally compact group $G$ is amenable if there is a left invariant mean on $L^\infty(G)$. It is well-known that amenability is equivalent to the existence of a sequence positive definite functions in $C_c(G)$ that converges to $1$ uniformly on compact subsets of $G$. Another characterization of amenability is the existence of a bounded approximate unit in the Fourier algebra $A(G)$.

All examples so far of groups satisfying \eqref{eq:AB0} are build from compact groups and solvable groups as semidirect products. In particular all such groups are amenable.
\begin{que}
Do there exist non-amenable groups $G$ satisfying \eqref{eq:AB0}?
\end{que}

%A potential candidate for such a group is $M(2,\Q_p)\rtimes\UL(2,\Q_p)$. Here the group $\UL(2,\Q_p)$ as defined above acts on the $2\times 2$ matrices $M_2(\Q_p)$ by acting on each column by matrix multiplication.

We should remark that there exist non-amenable AR-groups. An example was given in \cite{MR509261}, and the group $\Q_p^2\rtimes\UL(2,\Q_p)$ studied in Section~\ref{sec:AR-example} is also such an example. That the group $\Q_p^2\rtimes\UL(2,\Q_p)$ is non-amenable follows for instance from \cite[14.9]{MR767264}.

%\begin{que}
%Do there exist non-compact property (T) groups $G$ satisfying \eqref{eq:AB0}?
%\end{que}
%We remark that property (T) groups are unimodular, so examples, if they exist, should be sought among the disconnected groups, due to Baggett's theorem (Theorem~\ref{thm:Baggett}).

We show below that for groups satisfying \eqref{eq:AB0}, amenability is equivalent to the Haagerup property (see Proposition~\ref{prop:Haagerup}). Recall that a group has the Haagerup property \cite{MR1852148} if there is a net of positive definite functions $\phi_n\colon G\to\C$ such that $\phi_n\in C_0(G)$ and $\phi_n\to 1$ uniformly on compact subsets of $G$ as $n\to\infty$.

\begin{prop}\label{prop:Haagerup}
Let $G$ be a locally compact group satisfying $A(G) = B_0(G)$. Then $G$ is amenable if and only if $G$ has the Haagerup property.
\end{prop}
\begin{proof}
Amenable groups always have the Haagerup property so we only prove the converse. Assume $G$ has the Haagerup property. Then there is a net of positive definite functions $\phi_n\colon G\to\C$ such that $\phi_n\in C_0(G)$ and $\phi_n\to 1$ uniformly on compact subsets of $G$ as $n\to\infty$. One can assume that $\phi_n(1) = 1$ for all $n$. Since $A(G) = B_0(G)$, the functions $\phi_n$ belong to $A(G)$. From \cite[Lemma~10.3]{Nielson-thesis} (see also \cite[Proposition~5.1]{MR784292}) follows that $\phi_n$ is a bounded approximate identity in $A(G)$. Thus, $G$ is amenable by \cite{MR0239002}.

Alternative proof: This is based on \cite{MR3211015}. If $G$ satisfies $A(G) = B_0(G)$, then the regular representation $\lambda$ and the enveloping $C_0$-representation, denoted $\pi_0$, are quasi-equivalent. In particular, they are weakly equivalent so that $C^*_\lambda(G) = C^*_{\pi_0}(G)$. Thus,
\begin{align*}
G \text{ is amenable} &\iff C^*(G) = C^*_\lambda(G)
\\& \iff C^*(G) = C^*_{\pi_0}(G)
\\& \iff G \text{ has the Haagerup property}.
\qedhere
\end{align*}
\end{proof}

%\appendix
%\appendixpage
%\section{Basic facts about \texorpdfstring{$p$}{p}-adic matrix groups}

\section*{Appendix: Basic facts about \texorpdfstring{$p$}{p}-adic matrix groups}
\renewcommand{\theequation}{A.\arabic{equation}}

In the following, we state two basic facts about $p$-adic matrices. These facts are certainly known, but in lack of a reference we have included proofs.

Throughout, $p$ will denote a fixed prime. Let $\Z_p$ denote the $p$-adic integers and $\Q_p$ the $p$-adic field, which is the field of fractions of $\Z_p$. A good and elementary introduction to $p$-adic numbers is given in \cite{MR1760253}.

We denote the $p$-adic absolute value of $x\in\Q_p$ by $|x|_p$. The following partial converse of the ultrametric property is easy to verify. For $x,y\in\Q_p$ we have
\begin{align}\label{eq:ultrametric2}
|x+y|_p < \max\{|x|_p,|y|_p\} \implies |x|_p = |y|_p.
\end{align}

Let $\mu$ be a Haar measure on $\Q_p$.
%normalized such that $\mu(\Z_p) = 1$.
For $j\in\Z$, define the set $K_j = \{ x\in\Q_p \mid |x| \leq p^j \}$, and note that 
$$
K_j = \bigsqcup_{i=0}^{p-1} i + K_{j-1}.
$$
It follows by induction and left invariance of $\mu$ that
\begin{align}\label{eq:Kj}
\mu(K_j) = p^j\mu(K_0).
\end{align}
If $d\in\Q_p\setminus\{0\}$ with $|d|_p = p^j$ we see that $d K_0 = K_j$ and hence
\begin{align}\label{eq:modulus}
\mu(dK_0) = \mu(K_j) = p^j\mu(K_0) = |d|_p \mu(K_0).
\end{align}

\begin{lem}\label{lem:determinant}
Let $\mu$ be a Haar measure on $\Q_p^n$. For a Borel subset $A\subseteq\Q_p^n$ and a matrix $g\in\GL(n,\Q_p)$ we have
$$
\mu(gA) = |\det g\,|_p \mu(A).
$$
\end{lem}
\begin{proof}
If $\nu$ denotes the measure on $\Q_p^n$ defined as $\nu(C) = \mu(gC)$, then clearly $\nu$ is a Haar measure, and hence $\nu = \delta(g)\mu$ for some constant $\delta(g) > 0$. We will argue that $\delta(g) = |\det g\,|_p$.

Recall the Bruhat decomposition $\GL(n,\Q_p) = BWB$, where $B$ denotes the upper triangular matrices in $\GL(n,\Q_p)$ and $W$ is the group of $n\times n$ permutation matrices. Clearly, $\delta(gh) = \delta(g)\delta(h)$ when $g,h\in\GL(n,\Q_p)$, so by the Bruhat decomposition we need only show that $\delta(g) = |\det g\,|_p$, when $g\in B$ and when $g\in W$.

When $g\in W$ we obviously have $\delta(g) = 1 = |\pm 1|_p = |\det g\,|_p$. If $g$ is upper triangular we write $g$ as a product $g=dm$, where $d$ is diagonal and $m$ is upper triangular with $1$ in each diagonal entry. Note that $\delta(m)^n = \delta(m^n) = \delta(1) = 1$, so $\delta(m) = 1 = |\det m\,|_p$.

It remains to prove that $\delta(d) = |\det d\,|_p$ for diagonal $d\in\GL(n,\Q_p)$. If $n = 1$, then $\mu(d \Z_p) = |d|_p \mu(\Z_p)$ by \eqref{eq:modulus}. This establishes the formula when $n = 1$. The general case, when $n \geq 1$, follows from the case $n=1$ since Haar measure on $\Q_p^n$ is the product measure of the Haar measures on $\Q_p$.

This establishes the lemma.
\end{proof}

Let $\xi_1\in\hat{\Q_p}$ be the standard character on $\Q_p$ defined as follows. Write a $p$-adic number $\sum_{j\in\Z} c_j p^j \in\Q_p$ in the $p$-adic expansion, where ${c_j\in\{0,\ldots,p-1\}}$ and only finitely many $c_j$ with $j< 0$ are non-zero. Then
$$
\xi_1\left(\sum_{j\in\Z} c_j p^j \right) = \exp\left(2\pi i \sum_{j<0} c_jp^j\right).
$$
When $y\in\Q_p$, the map $\xi_y$ defined on $\Q_p$ as
$$
\xi_y(x) = \xi_1(xy), \qquad x\in\Q_p,
$$
is also a character on $\Q_p$, and any character of $\Q_p$ is of the form $\xi_y$ for a unique $y\in\Q_p$. This gives the usual isomorphism $\hat{\Q_p}\simeq\Q_p$ (see \cite[Theorem~4.12]{MR1397028}). We will use the bracket notation for the duality:
$$
\la x,y\ra = \xi_y(x) = \xi_1(xy).
$$
The dual $\hat{\Q_p^n}$ is identified with $\Q_p^n$ coordinate-wise.

The group $\GL(n,\Q_p)$ acts on $\Q_p^n$ by matrix multiplication. This induces a the dual action of $\GL(n,\Q_p)$ on the dual $\hat{\Q_p^n}$ of $\Q_p^n$ defined in the following way. For $g\in\GL(n,\Q_p)$ and $\chi$ a character on $\Q_p$ we get a character $g.\chi$ given by $(g.\chi)(x) = \chi(g^{-1}x)$. Under the usual isomorphism $\hat{\Q_p^n}\simeq \Q_p^n$ as above, we obtain the \emph{dual action} $\GL(n,\Q_p)\acts\Q_p^n$.

\begin{lem}\label{lem:dual-action}
The dual action $\GL(n,\Q_p)\acts\Q_p^n$ is determined by $g.x = (g^t)^{-1} x$ where $g\in\GL(n,\Q_p)$ and $x\in\Q_p^n$. Here $g^t$ denotes the transpose of $g$.
\end{lem}
\begin{proof}
We will use bracket notation for duality between elements of $\Q_p^n$. Let $g = (g_{ij})_{i,j=1}^n$ be an $n\times n$ matrix in $M_n(\Q_p)$. We first claim that
$$
\la gx,y\ra = \la x,g^t y \ra
$$
for every $x,y\in\Q_p^n$. Indeed,
\begin{align*}
\la gx,y \ra = \prod_{i=1}^n \la (gx)_i, y_i\ra = \prod_{i=1}^n\prod_{j=1}^n \la g_{ij} x_j, y_i\ra = \prod_{i,j=1}^n \xi_1(g_{ij} x_j y_i)
\end{align*}
and
\begin{align*}
\la x,g^ty\ra = \prod_{j=1}^n \la x_j , (g^t y)_j \ra = \prod_{j=1}^n \prod_{i=1}^n \la x_j, g^t_{ji} y_i \ra = \prod_{i,j=1}^n \xi_1(x_j g_{ij} y_i),
\end{align*}
which proves the claim. Now, for $g\in\GL(n,\Q_p)$ we have
$$
\la g.x , y\ra = \la x, g^{-1} y\ra = \la (g^{-1})^t x , y\ra,
$$
which proves that $g.x = (g^{-1})^t x = (g^t)^{-1} x$.
\end{proof}

\end{document}